\documentclass[12pt,reqno]{amsart}
\usepackage{amsmath,amsfonts,amssymb,amsthm,amscd,latexsym}

\usepackage{color}

\newtheorem{theorem}{Theorem}[section]
\newtheorem{lemma}[theorem]{Lemma}

\newtheorem{remark}[theorem]{Remark}

\begin{document}

\date{\today}

%
%
%
%
%
%
%
%
%
\title[2-Local automorphisms on $AW^\ast$-algebras]
 {2-Local automorphisms on $AW^\ast$-algebras}
\author[Shavkat Ayupov]{Shavkat Ayupov}

\address{%
 V.I.Romanovskiy Institute of Mathematics\\
  Uzbekistan Academy of Sciences, 81 \\ Mirzo Ulughbek street, 100170  \\
  Tashkent,   Uzbekistan}

 \address{National University of Uzbekistan, 4 \\ University str.,Tashkent, Uzbekistan}

\email{sh$_{-}$ayupov@mail.ru}

\author{Karimbergen Kudaybergenov}
\address{Ch. Abdirov 1 \\
Department of Mathematics \\
Karakalpak State University \\
Nukus 230113, Uzbekistan}

\email{karim2006@mail.ru}

\author{Turabay  Kalandarov}
\address{Ch. Abdirov 1 \\
Department of Mathematics \\
Karakalpak State University \\
Nukus 230113, Uzbekistan}

\email{turaboy$_-$kts@mail.ru}



\date{November 3, 2018}
\dedicatory{With the deep respect, we dedicate the article to the 65-th anniversary of Professor  Ben de Pagter.}

\begin{abstract}
The paper is devoted to  2-local automorphisms  on $AW^\ast$-algebras. Using the technique of matrix algebras over a unital Banach algebra we prove that any 2-local automorphism on an arbitrary $AW^\ast$-algebra without finite type~I direct summands is a global automorphism.
\end{abstract}

\maketitle

\section{Introduction and the Main Theorem}

\medskip

In 1990, Kadison \cite{Kad}  and Larson and Sourour \cite{Lar} independently
introduced the concept of a local derivation. A linear
map \(\Delta : \mathcal{A} \to \mathcal{M}\) is called a
\textit{local derivation} if  for every \(x\in\mathcal{A}\) there
exists a derivationû \(D_x\)  (depending on \(x\)) such that
\(\Delta(x) = D_x(x).\) It is natural to consider under
which conditions local derivations automatically become
derivations. Many partial results have been done in this problem.
In \cite{Kad} Kadison shows that every norm-continuous local
derivation from a von Neumann algebra \(M\) into a dual
\(M\)-bimodule is a derivation.
In~\cite{Joh} Johnson extends
Kadison's result and proves every local derivation from a
$C^{\ast}$-algebra $\mathcal{A}$ into any Banach
$\mathcal{A}$-bimodule is a derivation.

In 1997, \v{S}emrl \cite{Semrl97}  initiated the study of so-called  2-local derivations
and 2-local automorphisms on algebras. Namely, he  described such maps  on the algebra \(B(H)\) of all
bounded linear operators on an infinite dimensional separable Hilbert space  \(H\).

In the above notations,  a map \(\Delta : \mathcal{A} \to
\mathcal{A}\) (not necessarily linear) is called a \textit{2-local
automorphism} if, for every \(x,y \in  \mathcal{A},\) there exists an
automorphism \(\Phi_{x,y} : \mathcal{A} \to \mathcal{A}\) such that
\(\Phi_{x,y} (x) = \Delta(x)\) and \(\Phi_{x,y}(y) = \Delta(y).\)

 Afterwards  local derivations and 2-local derivations have
been investigated by many authors on different algebras and many
results have been obtained in \cite{AK2016JP, AK2016, AKP,
  Kad, KimKim04,  Semrl97}.

  In \cite{BFGP} it was established that every  2-local $\ast$-homomorphism
from a von Neumann algebra into a $C^\ast$-algebra is a linear
$\ast$-homomorphism. These authors also proved that every 2-local Jordan $\ast$-homomorphism from
 a JBW*-algebra into a JB*-algebra is  a
Jordan *-homomorphism.

In the present paper we extend the  result obtained in
\cite{AK2016JP}  for 2-local derivations on  $AW^\ast$-algebras  to the case of 2-local automorphisms on $AW^\ast$-algebras  .

If  $\Delta :\mathcal{A}\rightarrow \mathcal{A}$ is  a 2-local
automorphism, then from the definition it easily follows that
$\Delta$ is homogenous. At the same time,
\begin{equation*}\label{joor}
 \Delta(x^2)=\Phi_{x,x^2}(x^2)=\Phi_{x,x^2}(x)\Phi_{x,x^2}(x)=\Delta(x)^2
\end{equation*}
for each $x\in \mathcal{A}.$ This means that additive (and hence, linear) 2-local
automorphism  is a Jordan automorphism.

The following  Theorem is the main result of this paper.

\begin{theorem}\label{kaplanal} Let $M$ be an arbitrary
$AW^\ast$-algebra without finite type I direct summands. Then any 2-local automorphism
 $\Delta$ on  \(M\)
is an automorphism.
\end{theorem}

The proof of this Theorem is based on  representations of $AW^\ast$-algebras as matrix algebras over
a unital Banach algebra  with the following two properties:

\textbf{(J)}: \textit{for any Jordan automorphism  $\Phi$ on
\(\mathcal{A}\) there exists a decomposition  $\mathcal{A}=\mathcal{A}_1\oplus\mathcal{A}_2$
such that}
$$
x\in \mathcal{A}\mapsto p_1(\Phi(x))\in \mathcal{A}_1
$$
\textit{is a homomorphism and}
$$
x\in \mathcal{A}\mapsto p_2(\Phi(x))\in \mathcal{A}_2
$$
\textit{is an anti-homomorphism, where $p_i$ is a projection from $\mathcal{A}$ onto
$\mathcal{A}_i,$ $i=1,2$}

\textbf{(M)}: \textit{There exist elements $x, y \in \mathcal{A}$ such that $xy=0$ and $yx\neq 0.$}

\begin{remark}\label{remark}
 Note that if an algebra $\mathcal{A}$ contains  a subalgebra isomorphic to the matrix algebra $M_2(\mathbb{C}),$ then
it satisfies the condition \textbf{(M)}. Indeed, for  matrices $x=\left(
                                                          \begin{array}{cc}
                                                            0 & 1 \\
                                                            0 & 0 \\
                                                          \end{array}
                                                        \right)$ and
$y=\left(
                                                          \begin{array}{cc}
                                                            1 & 0 \\
                                                            0 & 0 \\
                                                          \end{array}
                                                        \right),$
we have $xy=0$ and $yx\neq 0.$
\end{remark}

\section{The proof of the main result}

The key  tool for  the proof of Theorem~\ref{kaplanal} is  the following.

\begin{theorem}\label{mainlocal}
Let $\mathcal{A}$ be  a unital  Banach algebra with the properties  \textbf{(J)} and \textbf{(M)} and
let $M_{2^n}(\mathcal{A})$ be the algebra of all $2^n \times
2^n$-matrices over $\mathcal{A},$ where \(n\geq 2.\) Then any
2-local automorphism  $\Delta$ on  $M_{2^n}(\mathcal{A})$  is an automorphism.
\end{theorem}

The proof  of Theorem~\ref{mainlocal} consists of two steps. In
the  first step  we shall show additivity of  $\Delta$ on the
subalgebra of diagonal matrices from $M_{n}(\mathcal{A}).$

Let $\{e_{i,j}\}_{i,j=1}^n$ be  the system of matrix units in
$M_n(\mathcal{A}).$ For $x \in M_n(\mathcal{A})$ by $x_{i,j}$ we
denote the $(i, j)$-entry of $x,$ where $1 \leq i, j \leq n.$ We
shall, if necessary, identify this element with the matrix from
$M_n(\mathcal{A})$ whose $(i,j)$-entry is $x_{i,j},$ other entries
are zero, i.e. $x_{i,j}=e_{i,i}xe_{j,j}.$

Each element \(x\in  M_n(\mathcal{A})\)  has the form
\[
x =\sum\limits_{i , j = 1}^n  x_{ij} e_{ij},\,\,  x_{ij}\in
\mathcal{A}, i , j \in  \overline{1, n}. \]

Let \(\psi: \mathcal{A} \to \mathcal{A}\) be an automorphism.
Setting
\begin{equation}\label{cender}
\overline{\psi}(x) =\sum\limits_{i , j = 1}^n  \psi(x_{ij})
e_{ij},\,\, x_{ij}\in \mathcal{A}, i , j \in  \overline{1, n}
\end{equation}
we obtain a well-defined linear operator \(\overline{\psi}\)
on  \(M_n(\mathcal{A}).\) Moreover
\(\overline{\psi}\) is an automorphism.

For an invertible element $a\in M_n(\mathcal{A})$ set
$$
\Phi_a(x)=axa^{-1},\, x\in M_n(\mathcal{A}).
$$
Then $\Phi_a$ is an automorphism and it is called a spatial automorphism.

It is known \cite[Corollary 3.14]{BO} that every automorphism  \(\Phi\)
on  \(M_n(\mathcal{A})\)  can be
represented as a product
\begin{equation}\label{decompos}
\Phi = \Phi_a\circ \overline{\psi},
\end{equation}
where \(\Phi_a\) is a spatial automorphism implemented by an
invertible element \(a \in  M_n(\mathcal{A}),\) while \(\overline{\psi}\)
is the automorphism  of the form \eqref{cender}  generated by an automorphism
 \(\psi\) on  \(\mathcal{A}.\)

Consider the following two matrices:
\begin{equation}\label{vv}
 u=\sum\limits_{i=1}^n \frac{1}{2^i}e_{i,i},\,
v=\sum\limits_{i=2}^n e_{i-1,i}.
\end{equation}

It is easy to see that an element $x \in M_n(\mathcal{A})$
commutes with $u$  if and only if it is diagonal, and if an
element  $a \in M_n(\mathcal{A})$ commutes with $v,$ then $a$ is
of the form
\begin{equation}\label{trian}
a=\left( \begin{array}{ccccccc}
a_1   & a_2   & a_3 & . & \ldots & a_n \\
0     & a_1   & a_2 & . & \ldots & a_{n-1}\\
0     & 0     & a_1 & . & \ldots & a_{n-2}\\
\vdots& \vdots& \vdots & \vdots  &\vdots & \vdots\\
0 & 0 &\ldots & . & a_1 & a_2\\
0 & 0 &\ldots & .& 0 & a_1
\end{array} \right).
\end{equation}

Further in Lemmata~\ref{lemmatwo}--\ref{lemmafive} we assume that
$n\geq 2.$

\begin{lemma}\label{lemmatwo}
For every $2$-local automorphism $\Delta$ on  $M_n(\mathcal{A})$
 there exists an automorphism   $\Phi$   such that
$\Delta|_{\mbox{sp}\{e_{i,j}\}_{i,j=1}^n}=\Phi|_{\mbox{sp}\{e_{i,j}\}_{i,j=1}^n},$
where $\mbox{sp}\{e_{i,j}\}_{i,j=1}^n$ is the linear span of the
set $\{e_{i,j}\}_{i,j=1}^n.$
\end{lemma}

\begin{proof}

Take an automorphism  \(\Phi_{u,v}\) on   \(M_n(\mathcal{A})\) such that
$$
\Delta(u)=\Phi_{u,v}(u),\, \Delta(v)=\Phi_{u,v}(v),
$$
where $u, v$ are the elements from~\eqref{vv}. Replacing $\Delta$
by $\Phi^{-1}_{u,v}\circ \Delta$, if necessary, we can assume that
$\Delta(u)=u, \Delta(v)=v.$

Let $i, j\in \overline{1, n}.$ Take an automorphism \(\Phi =
\Phi_a\circ \overline{\psi}\) of the form \eqref{decompos}
such that
$$
\Delta(e_{i,j})=a \overline{\psi}(e_{ij}) a^{-1},\,\Delta(u)=a \overline{\psi}(u) a^{-1}.
$$
Since $\Delta(u)=u$ and \(\overline{\psi}(u)=u,\)   it follows
that \([a, u]=0,\)  and therefore \(a\) has a diagonal form, i.e.
$a=\sum\limits_{s=1}^n a_{s}e_{s,s},\, a_s\in \mathcal{A},\, s\in
\overline{1, n}.$

In the same way, but starting with the element $v$ instead of $u$,
we obtain
$$
\Delta(e_{i,j})=be_{i,j}b^{-1},
$$
where  $b$  has the  form~\eqref{trian}, depending on $e_{i,j}.$
So
$$
\Delta(e_{i,j})=ae_{i,j}a^{-1}=b e_{i,j}b^{-1}.
$$
Since
$$
ae_{i,j}a^{-1}= a_{i}a_{j}^{-1}e_{i,j}
$$
 and
 $$
[b e_{i,j}b^{-1}]_{i,j}=1,
$$
 it follows that $\Delta(e_{i,j})=e_{i,j}.$

Now let us take a matrix $x=\sum\limits_{i,j=1}^n
\lambda_{i,j}e_{i,j}\in M_n(\mathbb{C}).$ Then
\begin{eqnarray*}
e_{j,i}\Delta(x)e_{j,i} & = & \Delta(e_{j,i})\Delta(x)\Delta(e_{j,i}) =
\Phi_{e_{j,i}, x} (e_{j,i}) \Phi_{e_{i,j}, x} (x) \Phi_{e_{j,i}, x} (e_{j,i}) =\\
& = & \Phi_{e_{j,i}, x} (e_{j,i}  x e_{j,i})=\Phi_{e_{j,i}, x} (\lambda_{i,j}  e_{j,i})=\\
&=& \lambda_{i,j}  \Phi_{e_{j,i}, x} (e_{j,i}) =\lambda_{i,j}
e_{j,i},
\end{eqnarray*}
i.e. $ e_{i,i}\Delta(x)e_{j,j}=\lambda_{i,j}e_{i,j}$ for all $i,j\in \overline{1,n}.$
This means that $\Delta(x)=x.$ The proof is complete.
\end{proof}

Further in Lemmata~\ref{three}--\ref{adj} we assume that $\Delta$
is a 2-local automorphism  on   $M_n(\mathcal{A})$ such that
$\Delta|_{\mbox{sp}\{e_{i,j}\}_{i,j=1}^n}= id|_{\mbox{sp}\{e_{i,j}\}_{i,j=1}^n}.$

Let $\Delta_{i,j}$ be the restriction of $\Delta$ onto
$\mathcal{A}_{i,j}=e_{i,i}M_n(\mathcal{A})e_{j,j},$ where $1 \leq
i, j \leq n.$

\begin{lemma}\label{three}
$\Delta_{i,j}$ maps $\mathcal{A}_{i,j}$ into itself.
\end{lemma}

\begin{proof}
Let us show that
\begin{equation}\label{compo}
\Delta_{i,j}(x) =e_{i,i}  \Delta(x) e_{j,j}
\end{equation}
for all $x\in \mathcal{A}_{i,j}.$

Take $x=x_{i,j}\in \mathcal{A}_{i,j},$  and consider an automorphism \(\Phi =
\Phi_a\circ \overline{\psi}\) of the form \eqref{decompos}
such that
$$
\Delta(x)=a\overline{\psi}(x)a^{-1},\,\Delta(u)=a\overline{\psi}(u)a^{-1},
$$
where \(u\) is the element from \eqref{vv}. Since $\Delta(u)=u$
and \(\overline{\psi}(u)=u,\) it follows that \([a, u]=0,\)  and
therefore \(a\) has a diagonal form. Then \(\Delta(x) =
a_{i}  \psi(x_{i j}) a_{j}^{-1} e_{i j}.\) This means that
\(\Delta(x) \in \mathcal{A}_{i,j}.\) The proof is complete.
\end{proof}

\begin{lemma}\label{lemmafour}
Let
$x=\sum\limits_{i=1}^n x_{i,i}$ be a diagonal matrix. Then
\begin{equation}\label{kkkk}
e_{k,k}\Delta(x)e_{k,k}=\Delta(x_{k,k})
\end{equation} for all $k\in
\overline{1,n}.$
\end{lemma}

\begin{proof}
Take an automorphism  \(\Phi\) of the form
\eqref{decompos}  such that
\begin{center}
\(\Delta(x)=a\overline{\psi}(x)a^{-1}\) and
\(\Delta(x_{k,k})=a\overline{\psi}(x_{kk})a^{-1}.\)
\end{center}
If  necessary, replacing $x_{k,k}$ by $\lambda e +x_{k,k}$ ($|\lambda|> ||x_{k,k}||$) we can assume that $x_{k,k}$ is invertible.
Using the equality \eqref{compo}, we obtain that $\Delta(x_{k,k})\in \mathcal{A}_{k,k}.$ Since
$\Delta(x_{k,k})a=a\overline{\psi}(x_{kk}),$
\begin{eqnarray*}
0 & = & (\Delta(x_{k,k})a)_{k,i} = x_{k,k} a_{k,i}, \\
0 & = & (a\overline{\psi}(x_{k,k}))_{i,k} =a_{i,k} \overline{\psi} (x_{k,k})
\end{eqnarray*}
for all $i\neq k.$ Since $x_{k,k}$ and $\overline{\psi}(x_{k,k})$ are invertible, we have that
$a_{i,k}=a_{k,i}=0$ for all $i\neq k.$ Further
\begin{eqnarray*}
\Delta(x_{k,k}) & = & e_{k,k}\Delta(x_{k,k})e_{k,k}=e_{k,k} a
\overline{\psi} (x_{k,k}) a^{-1}e_{k,k} =a_{k,k} \overline{\psi} (x_{k,k}) a_{k,k}^{-1}.
\end{eqnarray*}
Since  $x$ is a diagonal matrix and $a_{i,k}=a_{k,i}=0$ for all $i\neq k.$ we get
\begin{eqnarray*}
e_{k,k}\Delta(x)e_{k,k} & = &  e_{k,k} a
\overline{\psi} (x) a^{-1}e_{k,k} =a_{k,k} \overline{\psi} (x_{k,k}) a_{k,k}^{-1}.
\end{eqnarray*}
Thus $e_{k,k}\Delta(x)e_{k,k}=\Delta(x_{k,k}).$ The proof is
complete.
\end{proof}

\begin{lemma}\label{lemmafive}
Let  $x=x_{i,i}\in
\mathcal{A}_{i,i}.$ Then
\begin{equation}\label{jiji}
e_{j,i}\Delta(x)e_{i,j}=\Delta(e_{j,i}xe_{i,j})
\end{equation}
for every \(j\in \{1, \cdots, n\}.\)
\end{lemma}

\begin{proof}
The case when $i=j$
has been already proved (see
Lemma~\ref{lemmafour}).

Suppose that $i\neq j.$  For an arbitrary element $x=x_{i,i}\in
\mathcal{A}_{i,i},$  consider  $y=x+e_{j,i}xe_{i,j}\in
\mathcal{A}_{i,i}+\mathcal{A}_{j,j}.$ Take an automorphism  \(\Phi\) of the form
\eqref{decompos}  such that
\begin{center}
\(\Delta(y)=a\overline{\psi}(y)a^{-1}\) and
\(\Delta(v)=a\overline{\psi}(v)a^{-1},\)
\end{center}
where $v$ is the element from~\eqref{vv}. Since $\Delta(v)=v$ and
\(\overline{\delta}(v)=v,\) it follows that $a$ has the
form~\eqref{trian}. By Lemma~\ref{lemmafour} we obtain that
\begin{eqnarray*}
e_{j,i}\Delta(x)e_{i,j} & = &
e_{j,i}e_{i,i}\Delta(y)e_{i,i}e_{i,j} = a_1\overline{\psi}(y) a_1^{-1} e_{j,j},\\
 \Delta(e_{j,i}xe_{i,j}) & = &
e_{j,j}\Delta(y)e_{j,j}=a_1\overline{\psi}(x)a_1^{-1} e_{j,j}.
\end{eqnarray*}
The proof is complete.
\end{proof}

Further in Lemmata~\ref{six}--\ref{cc} we assume that $n\geq 3.$

\begin{lemma}\label{six}
 $\Delta_{i,i}$ is additive for all $i\in \overline{1,n}.$
\end{lemma}

\begin{proof}
Let $i\in \overline{1,n}.$ Since $n\geq 3,$ we
can take different numbers $k, s$ such that \linebreak
$(k-i)(s-i)\neq 0.$

For arbitrary $x, y\in \mathcal{A}_{i,i}$  consider the diagonal
element $z\in
\mathcal{A}_{i,i}+\mathcal{A}_{k,k}+\mathcal{A}_{s,s}$ such that
\(z_{ i,i} = x+y,\, z_{k, k} = x,\, z_{s, s} = y.\) Take an automorphism  \(\Phi\) of the form
\eqref{decompos}  such that
\begin{center}
\(\Delta(z)=a\overline{\psi}(z)a^{-1}\) and
\(\Delta(v)=a\overline{\psi}(v)a^{-1},\)
\end{center}
where $v$ is  the element from~\eqref{vv}. Since $\Delta(v)=v$ and
\(\overline{\delta}(v)=v,\) it follows that $a$ has the
form~\eqref{trian}. Using Lemmata~\ref{lemmafour} and
\ref{lemmafive} we obtain that
\begin{eqnarray*}
\Delta_{i,i}(x+y) &   \stackrel{\eqref{kkkk}}{=}  &
e_{i,i}\Delta(z)e_{i,i}= a_1 \overline{\psi}(x+y) a_1^{-1} e_{i,i},   \\
\Delta_{i,i}(x) &  \stackrel{\eqref{jiji}}{=}  & e_{i,k}\Delta(e_{k,i}x
e_{i,k})e_{k,i}\stackrel{\eqref{kkkk}}{=} e_{i,k}e_{k,k}\Delta(z)e_{k,k}e_{k,i}=\\
&  =  & a_1 \overline{\psi}(x) a_1^{-1} e_{i,i},\\
\Delta_{i,i}(y) &  \stackrel{\eqref{jiji}}{=} &
e_{i,s}\Delta(e_{s,i}y
e_{i,s})e_{s,i}\stackrel{\eqref{kkkk}}{=} e_{i,s}e_{s,s}\Delta(z)e_{s,s}e_{s,i}=\\
& = & a_1 \overline{\psi}(y) a_1^{-1} e_{i,i}.
\end{eqnarray*} Hence
$$
\Delta_{i,i}(x+y)=\Delta_{i,i}(x)+\Delta_{i,i}(y).
$$
The proof is complete.
\end{proof}

As it was mentioned in the beginning of the section  any additive
2-local automorphism  is a Jordan automorphism. Since
$\mathcal{A}_{i,i}\cong \mathcal{A}$ has the property
\textbf{(J)}, by Lemma~\ref{six} there exists a decomposition  $\mathcal{A}=\mathcal{A}_1\oplus\mathcal{A}_2$
such that
$$
x\in \mathcal{A}\mapsto p_1(\Delta_{i,i}(x))\in \mathcal{A}_1
$$
is a homomorphism and
$$
x\in \mathcal{A}\mapsto p_2(\Delta_{i,i}(x))\in \mathcal{A}_2
$$
is an anti-homomorphism.

Suppose that $p_2\neq 0.$ By the condition \textbf{(M)} we can find elements $x, y \in \mathcal{A}$ such that $xy=0$ and $yx\neq 0.$
Then
$$
0=p_2(\Delta_{i,i}(xy))=p_2(\Delta_{i,i}(y))p_2(\Delta_{i,i}(x)).
$$
On the other hand,
$$
\Delta_{i,i}(y)\Delta_{i,i}(x)=\Phi_{x,y}(y)\Phi_{x,y}(x)=\Phi_{x,y}(yx)\neq 0.
$$
From this contradiction we obtain that $p_2=0.$ So, we have the following

\begin{lemma}\label{lemmaseven}
 $\Delta_{i,i}$ is an automorphism  for all $i\in \overline{1,n}.$
\end{lemma}

Denote by $\mathcal{D}_n(\mathcal{A})$ the set of all diagonal
matrices from $M_n(\mathcal{A}),$ i.e. the set of all matrices of
the following form
\[
x=\left( \begin{array}{cccccc}
x_1 & 0 & 0 & \ldots & 0 \\
0 & x_2 &  0 & \ldots & 0\\
\vdots& \vdots& \vdots &\vdots & \vdots\\
0 & 0 &\ldots & x_{n-1} & 0\\
0 & 0 &\ldots & 0 & x_n
\end{array} \right).
\]

Let us consider an operator $\overline{\Delta_{1,1}}$ of the
form~\eqref{cender}. By Lemmata~\ref{lemmafour} and
\ref{lemmafive} we obtain that

\begin{lemma}\label{adj}
$\Delta|_{\mathcal{D}_n(\mathcal{A})}=\overline{\Delta_{1,1}}|_{\mathcal{D}_n(\mathcal{A})}$
and $\overline{\Delta_{1,1}}|_{\mbox{sp}\{e_{i,j}\}_{i,j=1}^n}=id|_{\mbox{sp}\{e_{i,j}\}_{i,j=1}^n}.$
\end{lemma}

Now we are in position to pass to the second step of our proof. In
this step we show that if a 2-local automorphism  $\Delta$ satisfies
the following conditions
\begin{center}
\(\Delta|_{\mathcal{D}_n(\mathcal{A})}\equiv id|_{\mathcal{D}_n(\mathcal{A})}\) and
\(\Delta|_{\mbox{sp}\{e_{i,j}\}_{i,j=1}^n}\equiv id|_{\mbox{sp}\{e_{i,j}\}_{i,j=1}^n},\)
\end{center}
then it is the identical map.

 In following five Lemmata~\ref{ss}-\ref{twotwo}  we
shall consider 2-local automorphisms which satisfy the latter equalities.

We denote by $e$ the unit of the algebra $\mathcal{A}.$

\begin{lemma}\label{ss}
Let $x\in M_n(\mathcal{A}).$ Then $\Delta(x)_{k,k}=x_{k,k}$ for all
\(k\in \overline{1,n}.\)
\end{lemma}

\begin{proof}
Let $x\in M_n(\mathcal{A}),$ and fix \(k\in \overline{1,n}.\)
Since $\Delta$ is
homogeneous, we can assume that $\|x_{k,k}\|<1,$ where $\|\cdot\|$
is the norm on $\mathcal{A}.$ Take a diagonal element $y$ in
\(M_n(\mathcal{A})\) with \(y_{k,k}=e+x_{k,k}\) and \(y_{i,i}=0\)
otherwise.  Since $\|x_{k,k}\|<1,$ it follows that $e+x_{k,k}$ is
invertible in $\mathcal{A}.$ Take an automorphism  \(\Phi\) of the form
\eqref{decompos}  such that
\begin{center}
\(\Delta(x)=a\overline{\psi}(x)a^{-1}\) and
\(\Delta(y)=a\overline{\psi}(y)a^{-1}.\)
\end{center}
Since  $y\in \mathcal{D}_n(\mathcal{A})$ we have that
$y=\Delta(y)=a\overline{\psi}(y)a^{-1},$ and therefore
\begin{eqnarray*}
0 &  = &\Delta(y)_{i,k}=a_{i,k}(e +x_{k,k}),\\
0 &  = & \Delta(y)_{k,i}=-(e +x_{k,k}) a_{k,i}
\end{eqnarray*}
for all \(i\neq k.\) Thus
$$
a_{i,k}=a_{k,i}=0
$$
for all \(i\neq k.\) The above equalities imply that
$$
\Delta(x)_{k,k}=\Delta(y)_{k,k}=x_{k,k}.
$$
The proof is complete.
\end{proof}

\begin{lemma}\label{zerro}
Let   $x$ be a matrix with $x_{k,s}=\lambda e.$  Then $\Delta(x)_{k,s}=\lambda e.$
\end{lemma}

\begin{proof}
We have
\begin{eqnarray*}
e_{s,k}\Delta(x)e_{s,k} & = & \Delta(e_{s,k})\Delta(x)\Delta(e_{s,k})
=\Phi_{e_{s,k}, x}(e_{s,k})\Phi_{e_{s,k}, x} (x) \Phi_{e_{s,k}, x}(e_{s,k})  =\\
& = & \Phi_{e_{s,k}, x} (e_{s,k} x e_{s,k})=\Phi_{e_{s,k}, x} (\lambda e_{s,k})=\lambda \Delta(e_{s,k})=\lambda e_{s,k}.
\end{eqnarray*}
Thus
$$
e_{k,k}\Delta(x)e_{s,s}=e_{k,s}e_{s,k}\Delta(x)e_{s,k}e_{k,s}=\lambda e_{k,s}.
$$
This means that \(\Delta(x)_{k,s}=\lambda e.\) The proof is complete.
\end{proof}

\begin{lemma}\label{cc}
Let \(k, s\) be numbers such that \(k\neq s\) and let
 $x$ be a matrix with $x_{k,s}=\lambda e,$ $\lambda \neq 0.$ Then $\Delta(x)_{s,k}=x_{s,k}.$
\end{lemma}

\begin{proof}
Take a diagonal
element \(y\) such that \(y_{k,k}=x_{s,k}\) and
\(y_{i,i}=\lambda_i e\) otherwise, where \(\lambda_i\,\, (i\neq
k)\) are distinct numbers with \(|\lambda_i|>\|x_{s,k}\|.\)  Take
an automorphism \(\Phi\) such that
\begin{center}
$\Delta(x)=\Phi(x)$ and $\Delta (y)=\Phi(y).$
\end{center}
Then $ya=a\overline{\psi}(y),$ and therefore
\begin{eqnarray*}
  0 & = & (ya-a\overline{\psi}(y))_{ij} = \lambda_j a_{i,j}-\lambda_i a
_{i,j}=a_{i,j}(\lambda_j -\lambda_i)\,\,\, \textrm{for}\,\,\, (i-j)(i-k)(j-k)\neq 0,\\
 0 & = & (ya-a\overline{\psi}(y))_{i,k} = a_{i,k}\overline{\psi}(y_{k,k})-\lambda_i a_{i,k}
=a_{i,k}(\overline{\psi}(x_{s, k})-\lambda_i)\,\,\, \textrm{for}\,\,\,  i\neq k,\\
0 & = &  (ya-a\overline{\psi}(y))_{k,j} = a_{k,j}\lambda_{j}-\overline{\psi}(y_{kk}) a_{k j}
=(\lambda_j-\overline{\psi}(x_{s, k}))a_{k,j} \,\,\, \textrm{for}\,\,\,  j\neq k.
\end{eqnarray*}
Thus
 \(a_{i,j}=0\) for all \(i\neq j,\) i.e.  \(a\) is a diagonal
element. Since
\[
\lambda e=\Delta(x)_{ks}=a_{kk} \lambda  ea_{ss}^{-1},
\]
it follows that \(a_{k,k}=a_{s,s}.\) Finally,
\begin{eqnarray*}
\Delta(x)_{s,k} & = &
a_{s,s}\overline{\psi}(x_{s,k})a_{k,k}^{-1}=\\
& = &
a_{k,k}\overline{\psi}(y_{k,k})a_{k,k}^{-1}=\Delta(y)_{k,k}=x_{s,k}.
\end{eqnarray*}
The proof is complete. \end{proof}

In the next two Lemmata  we assume that  $\Delta$ is a
2-local automorphism  on $M_{2}(\mathcal{A}).$

\begin{lemma}\label{z}
 Let
$x=\left(
     \begin{array}{cc}
       x_{1,1} & \lambda e \\
       x_{2,1} & x_{2,2} \\
     \end{array}
   \right)
$ and $y=\left(
     \begin{array}{cc}
       x_{1,1} & x_{1,2} \\
       x_{2,1} & x_{2,2} \\
     \end{array}
   \right),
$ where $|\lambda|> ||\Delta(y)_{1,2}||.$
Then
$\Delta(x)_{2,1}=\Delta(y)_{2,1}.$
\end{lemma}

\begin{proof}
Take an automorphism  \(\Phi\)  such that
\begin{center}
$\Delta(x)=\Phi(x)$ and $\Delta(y)=\Phi(y).$
\end{center}
Then
\begin{eqnarray*}
\left(
     \begin{array}{cc}
       0 & \lambda e-\Delta(y)_{1,2} \\
       \left(\Delta(x)-\Delta(y)\right)_{2,1} & 0 \\
     \end{array}
   \right)   \left(
     \begin{array}{cc}
       a_{1,1} & a_{1,2} \\
       a_{2,1} & a_{2,2} \\
     \end{array}
   \right) = \left(
     \begin{array}{cc}
       a_{1,1} & a_{1,2} \\
       a_{2,1} & a_{2,2} \\
     \end{array}
   \right)\left(
     \begin{array}{cc}
       0 & \lambda e-x_{1,2} \\
       0 & 0 \\
     \end{array}
   \right).
\end{eqnarray*}
Thus
$$
\left\{
  \begin{array}{ll}
    (\lambda  e-\Delta(y)_{1,2})a_{2,1}=0, & \hbox{} \\
    (\Delta(x)_{2,1}-\Delta(y)_{2,1})a_{1,1}=0. & \hbox{}
  \end{array}
\right.
$$
Since $|\lambda|> ||\Delta(y)_{1,2}||,$ it follows that $\lambda e-\Delta(y)_{1,2}$ is invertible in $\mathcal{A},$  and therefore
the first equality implies that
$a_{2,1}=0.$ Thus $a_{1,1}$ is invertible and the second equality gives us $\Delta(x)_{2,1}=
\Delta(y)_{2,1}.$
The proof is complete.
\end{proof}

\begin{lemma}\label{twotwo}
$\Delta=id.$
\end{lemma}

\begin{proof}
Let $x\in M_2(\mathcal{A}).$ By Lemma~\ref{ss} we have that $\Delta(x)_{k,k}=x_{k,k}$
for $k=1,2.$

Let now $k\neq s.$ Take a matrix  \(y\) with \(y_{s,k}=\lambda e\) and
\(y_{i,j}=x_{i,j}\) otherwise. By Lemma~\ref{cc} we have that
\(\Delta(y)_{k,s}=x_{k,s}.\) Further Lemma~\ref{z} implies that
$$
\Delta(x)_{k,s}=\Delta(y)_{k,s}=x_{k,s}.
$$
Thus $\Delta(x)_{k,s}=\Delta(y)_{k,s}=x_{k,s}$ for all $k,s=1,2,$ and therefore $\Delta(x)=x.$
The proof is complete.
\end{proof}

Now we are in position to prove
Theorem~\ref{mainlocal}.

\textit{Proof of Theorem~\ref{mainlocal}}.   Let $\Delta$ be a
2-local automorphism  on $M_{2^n}(\mathcal{A}),$ where $n\geq 2.$ By
Lemma~\ref{lemmatwo}   there exists an automorphism  $\Phi_1$ on
$M_{2^n}(\mathcal{A})$ such that
$\Delta|_{\mbox{sp}\{e_{i,j}\}_{i,j=1}^{2^n}}=\Phi_1|_{\mbox{sp}\{e_{i,j}\}_{i,j=1}^{2^n}}.$
 Replacing, if necessary, $\Delta$ by $\Phi_1^{-1}\circ\Delta,$ we may assume that
$\Delta$ is identical on $\mbox{sp}\{e_{i,j}\}_{i,j=1}^{2^n}.$
Further, by Lemma~\ref{adj}  there exists an automorphism $\Phi_2$ on $M_{2^n}(\mathcal{A})$ such that
$\Delta|_{\mathcal{D}_{2^n}}=\Phi_2|_{\mathcal{D}_{2^n}}.$
Now replacing $\Delta$ by $\Phi_2^{-1}\circ \Delta,$ we can
assume that $\Delta$ acts as the identity on  $\mathcal{D}_{2^n}.$ So, we
can assume that
\begin{center}
$\Delta|_{\mbox{sp}\{e_{i,j}\}_{i,j=1}^{2^n}}\equiv id|_{\mbox{sp}\{e_{i,j}\}_{i,j=1}^{2^n}}$ and
$\Delta|_{\mathcal{D}_{2^n}}\equiv id|_{\mathcal{D}_{2^n}}.$
\end{center}

Let us to show that $\Delta\equiv id.$ We proceed by induction on
$n.$

Let $n=2.$ We identify  the algebra $M_4(\mathcal{A})$ with the
algebra of $2\times 2$-matrices   $M_2(\mathcal{B}),$ over
$\mathcal{B}=M_2(\mathcal{A}).$

Let $\{e_{i,j}\}_{i,j=1}^4$ be a system of matrix units in
$M_4(\mathcal{A}).$ Then
$$
p_{1,1}=e_{1,1}+e_{2,2},\, p_{2,2}=e_{3,3}+e_{4,4},\,
p_{1,2}=e_{1,3}+e_{2,4},\, p_{2,1}=e_{3,1}+e_{4,2}
$$
is the system of matrix units in $M_2(\mathcal{B}).$ Since
$\Delta|_{\mbox{sp}\{e_{i,j}\}_{i,j=1}^{4}}\equiv id|_{\mbox{sp}\{e_{i,j}\}_{i,j=1}^{4}},$ it follows
that $\Delta|_{\mbox{sp}\{p_{i,j}\}_{i,j=1}^{2}}\equiv id|_{\mbox{sp}\{p_{i,j}\}_{i,j=1}^{2}}.$

 Take an arbitrary element $x\in
p_{1,1}M_2(\mathcal{B})p_{1,1}\equiv \mathcal{B}.$ Choose an automorphism $\Phi$ on $M_2(\mathcal{B})$ such that
$$
\Delta(x)=\Phi(x),\, \Delta(p_{1,1})=\Phi(p_{1,1}).
$$
Since  $\Delta(p_{1,1})=p_{1,1},$ we obtain that
$$
p_{1,1}\Delta(x)p_{1,1}=p_{1,1}\Phi(x)p_{1,1}=\Delta(x).
$$
This means that the restriction $\Delta_{1,1}$ of $\Delta$ onto
$p_{1,1}M_2(\mathcal{B})p_{1,1}\equiv \mathcal{B}$ maps
$\mathcal{B}=M_2(\mathcal{A})$ into itself.

If $\mathcal{D}_4$ is the subalgebra of diagonal matrices from $M_4(\mathcal{A}),$ then
$p_{1,1}\mathcal{D}_4p_{1,1}$ is the  subalgebra  of diagonal matrices in the algebra
$M_2(\mathcal{A}).$ Since $\Delta|_{\mathcal{D}_{4}}\equiv id|_{\mathcal{D}_{4}},$ it
follows that $\Delta_{1,1}$ acts identically on  diagonal matrices
from  $M_{2}(\mathcal{A}).$ So,
\begin{center}
$\Delta_{1,1}|_{\mbox{sp}\{e_{i,j}\}_{i,j=1}^{2}}\equiv id|_{\mbox{sp}\{e_{i,j}\}_{i,j=1}^{2}}$ and
$\Delta_{1,1}|_{p_{1,1}\mathcal{D}_4p_{1,1}}\equiv id|_{p_{1,1}\mathcal{D}_4p_{1,1}}.$
\end{center}
By Lemma~\ref{twotwo} it follows that $\Delta_{1,1}\equiv id.$

Let $\mathcal{D}_2$ be the set of diagonal matrices from  $M_2(\mathcal{B}).$ Since
$$
\mathcal{D}_2=\left( \begin{array}{cc}
\mathcal{B} & 0 \\
0 & \mathcal{B}
\end{array} \right)
$$
and $\Delta_{1,1}=id,$ Lemma~\ref{lemmafour} implies  that
$\Delta|_{\mathcal{D}_2}\equiv id|_{\mathcal{D}_2}.$ Hence, $\Delta$ is a 2-local
derivation on $M_2(\mathcal{B})$ such that
\begin{center}
$\Delta|_{\mbox{sp}\{p_{i,j}\}_{i,j=1}^{2}}\equiv id|_{\mbox{sp}\{p_{i,j}\}_{i,j=1}^{2}}$ and
$\Delta|_{\mathcal{D}_2}\equiv id|_{\mathcal{D}_2}.$
\end{center}
Again by Lemma~\ref{twotwo} it
follows that $\Delta\equiv id.$

Now assume that the assertion of the Theorem is true for $n-1.$

Considering the algebra   $M_{2^n}(\mathcal{A})$ as the algebra of
$2\times 2$-matrices  $M_2(\mathcal{B})$ over
$\mathcal{B}=M_{2^{n-1}}(\mathcal{A})$ and repeating the above
arguments we obtain that  $\Delta\equiv id.$ The proof is complete.
$\Box$

Now we apply Theorem~\ref{mainlocal} to the proof of our main result which
describes  2-local automorphism  on  $AW^\ast$-algebras.

First note that by  \cite[Theorem 3.3]{Stormer} (see also \cite[Theorem  3.2.3]{BR}) any $C^\ast$-algebra, in particular, $AW^\ast$-algebra,  has the
property \textbf{(J)}.

\textit{Proof of Theorem~\ref{kaplanal}}.
Let   $M$ be  an arbitrary $AW^\ast$-algebra without finite type I
direct summands. Then there exist mutually orthogonal central projections $z_1, z_2, z_3$  in $M$ such that
$M=z_1M \oplus z_2 M \oplus z_3 M,$ where $z_1 M, z_2 M, z_3 M$ are algebras of types I$_\infty,$
II and III, respectively. Then  the halving Lemma~\cite[P. 120, Theorem 1]{Berber}
applied to each summand implies that the unit $z_i$ of the algebra $z_i M, $
($i=1,2,3$)  can be
represented as a sum of mutually equivalent orthogonal projections
$e_1^{(i)}, e_2^{(i)}, e_3^{(i)},e_4^{(i)}$ from $z_iM.$ Set $e_k=\sum\limits_{i=1}^3 e_k^{(i)},$ $k=1,2,3,4.$
Then the map $x\mapsto
\sum\limits_{i,j=1}^4 e_ixe_j$ defines an  isomorphism between the
algebra $M$ and the matrix algebra $M_4(\mathcal{A}),$ where
$\mathcal{A}=e_{1,1}Me_{1,1}.$ Moreover, the algebra \(\mathcal{A}\) has the properties
\textbf{(J)} and \textbf{(M)} (see the Remark~\ref{remark} after the definition of property \textbf{(M)}).
Therefore Theorem~\ref{mainlocal} implies that any
2-local automorphism  on   $M$  is an  automorphism. The
proof is complete.
$\Box.$

\section*{Acknowledgments} The authors are indebted to the reviewer for useful remarks.


\begin{thebibliography}{22}






\bibitem{AK2016JP} Sh.~A.~Ayupov and K.~K.~Kudaybergenov, \textit{2-Local derivations on
matrix algebras over semi-prime Banach algebras and on
\(AW^\ast\)-algebras,} Journal of Physics: Conference Series,
\textbf{697} (2016) 1--10.


\bibitem{AK2016} Sh. A. Ayupov, K. K. Kudaybergenov, \textit{Derivations, local
and 2-local derivations on algebras of measurable operators,} in
Topics in Functional Analysis and Algebra, \emph{Contemporary
Mathematics}, vol. 672, Amer. Math. Soc., Providence, RI, 2016,
pp. 51-72.



\bibitem{AKP}
Sh. A. Ayupov,  K. K. Kudaybergenov and  A. M. Peralta,\textit{ A
survey on local and 2-local derivations on C$^\ast$- and von
Neumann algebras,} in
Topics in Functional Analysis and Algebra, \emph{Contemporary
Mathematics}, vol. 672, Amer. Math. Soc., Providence, RI, 2016,
pp. 73-126.


\bibitem{BO}
G.M.Benkart, J.M. Osborn,
\textit{Derivations and automorphisms of nonassociative matrix algebras,}
Trans. AMS.  \textbf{263}  2 (1981) 411--430.










\bibitem{Berber}  S. Berberian, Bear *-rings, Springer 1972, 2nd edition 2011.



\bibitem{BFGP} M.J. Burgos, F.J. Fernandez Polo, J.J. Garces, A.M. Peralta,
\textit{A Kowalski-Slodkowski theorem for 2-local $\ast$-homomorphisms
on von Neumann algebras}, Revista Serie A Matematicas \textbf{109}, Issue 2 (2015), Page 551-568.


\bibitem{BR} O. Brattelli, D. Robinson, \textit{Operator algebras and quantum statistical mechanics,}
2nd Edition Springer-Verlag Berlin Heidelberg New York 2002.















\bibitem{Joh}  {B.~E.~Johnson,} \textit{Local derivations on $C^{\ast}$-algebras
are derivations,} Trans. Amer. Math. Soc., \textbf{353} (200)
313--325.




\bibitem{Kad}{R.~V.~Kadison,} \textit{Local derivations,}  J.~Algebra, \textbf{130}
(1990)  494--509.



\bibitem{KimKim04} S.O. Kim, J.S. Kim,
\textit{Local automorphisms and derivations on $\mathbb{M}_n$,}
Proc. Amer. Math. Soc. \textbf{132}, no. 5, 1389-1392~(2004).




\bibitem{Lar}{D.~R.~Larson and  A.~R.~Sourour,} \textit{Local derivations and local automorphisms of
$B(X)$,} Operator theory: operator algebras and applications, part
2 (Durham,NH, 1988), 187--194, Proc. ~Sympos.~Pure~Math. ~51, Part
2, Amer.Math.Soc.,~Providence,~RI,~(1990).




\bibitem{Semrl97}  P. \v{S}emrl, \textit{Local automorphisms and derivations
on $B(H)$,} Proc. Amer. Math. Soc. \textbf{125}, 2677-2680 (1997).



\bibitem{Stormer} E. Stormer, \textit{On the Jordan structure of $C^\ast$-algebras},
Trans. Amer. Math. Soc. \textbf{120}  (1965), 438-447.


\end{thebibliography}
\end{document}